\documentclass[12pt, reqno]{amsart}

\usepackage[all, arc]{xy}
\usepackage{enumerate}
\usepackage{mathrsfs}
\usepackage{hyperref}
\usepackage{stmaryrd}
\usepackage{scalerel}
\usepackage{comment}
\usepackage[all]{xy}
\usepackage{amsthm}
\usepackage{tipa}
\usepackage{verbatim}
\usepackage{tikz-cd}
\usepackage{multirow}

\usepackage[latin1]{inputenc}
\usepackage{babel}
\usepackage{amsmath}
\usepackage[T1]{fontenc}
\usepackage[T3,OT2,T1]{fontenc}
\usepackage{textcomp}
\usepackage{bm}
\addto{\captionsafrikaans}{}
\addto{\captionsenglish}{}
\usepackage{graphicx}
\usepackage{color}
\usepackage{eso-pic}
\usepackage{amssymb}
\usepackage{mathtools}
\usepackage{multicol}
\usepackage{paralist}
\usepackage{geometry}
\usepackage{algorithmic}
\usepackage[ruled,vlined]{algorithm2e}
\usepackage{amsfonts}
\usepackage{filecontentsdef}
\usepackage{xcolor}
\usepackage{tikz-3dplot}
\usepackage{url}
\usepackage{graphicx}
\usetikzlibrary{decorations.pathmorphing}
\usepackage{cooltooltips}
\usepackage{setspace}

\DeclareOldFontCommand{\bf}{\normalfont\bfseries}{\mathbf}

\newtheorem{theorem}{Theorem}[section]
\newtheorem{corollary}[theorem]{Corollary}
\newtheorem{lemma}[theorem]{Lemma}
\theoremstyle{definition}
\newtheorem{definition}[theorem]{Definition}
\newtheorem{remark}[theorem]{Remark}
\newtheorem{example}[theorem]{Example}

\newcommand{\introthmname}{}
\newtheorem{introthminn}{\introthmname}

\newtheorem{introthminn1}{\introthmname}

\newcommand{\sophie}[1]{{\color{red}\textsf{$\clubsuit\clubsuit\clubsuit$}  Sophie: [#1]}}
\newcommand{\eliza}[1]{{\color{blue}\textsf{$ \spadesuit \spadesuit \spadesuit$}  eliza: [#1]}}

\newcommand{\nn}[1]{ \llbracket 1,  #1 \rrbracket}

\newcommand{\oid}[1]{\operatorname{id}}

\newcommand{\tr}[1]{\mathfrak{t}_{\zeta_{#1}}^{\pm}}
\newcommand{\ts}[1]{\mathfrak{t}_{\!\zeta_{\! #1}}^{+}}
\newcommand{\tq}[1]{\mathfrak{t}_{\! \zeta_{\! #1}}^{-}}

\newcommand\restr[2]{{
		\left.\kern-\nulldelimiterspace 
		#1 
		\vphantom{\big|} 
		\right|_{#2} 
}}



\title{\bfseries When is a 2-Power Cyclotomic Extension cyclic?}
\author{Sophie Marques and Elizabeth Mrema}

\begin{document}
	
	\setcounter{tocdepth}{3}
	\maketitle
	\begin{center}
		\rm e-mail: smarques@sun.ac.za
		
		\it
		Department of Mathematical Sciences, 
		Stellenbosch University, \\
		Stellenbosch, 7600, 
		South Africa\\
		\&
		NITheCS (National Institute for Theoretical and Computational Sciences), \\
		South Africa \\  \bigskip
		
		\rm e-mail: 25138413@sun.ac.za
		
		\it
		Department of Mathematical Sciences, 
		Stellenbosch University, \\
		Stellenbosch, 7600, 
		South Africa
	\end{center} 
	\tableofcontents

	\begin{abstract}
	This paper characterizes the cyclicity property of $2$-power cyclotomic extensions through various means: the structure of the Galois groups, the nature of their subextensions, tower decompositions, and, most importantly, specific conditions on the base field. \\
	
		
		Galois group, tower.
		\vspace*{\fill}
		
	\end{abstract}

	\section*{Introduction}

Given a positive integer \( e \) and a field \( F \), the Galois group of the cyclotomic extension \( F(\zeta_{2^e})/F \), referred to as a 2-power cyclotomic extension, is isomorphic to a subgroup of the unit group \(\left(\mathbb{Z}/2^e\mathbb{Z}\right)^\times\) (see \cite[Theorem 2.3]{conradcyclotomic}). For \( e > 2 \), this group \(\left(\mathbb{Z}/2^e\mathbb{Z}\right)^\times\) is not cyclic (see \cite[Chapter 4, Theorem 2']{kenneth-classical}), indicating that the $2$-power cyclotomic extension is not always cyclic. For example, over the field \(\mathbb{Q}\), the Galois group of the cyclotomic extension \(\mathbb{Q}(\zeta_{16})/\mathbb{Q}\) is isomorphic to \(\left(\mathbb{Z}/16\mathbb{Z}\right)^\times\), which is isomorphic to the non-cyclic group \(\mathbb{Z}/2\mathbb{Z} \times \mathbb{Z}/4\mathbb{Z}\). The structure of this extension in terms of tower decompositions is as follows:
	\begin{center}
	\scalebox{0.7}{	\begin{tikzpicture}[scale=0.8] 
			\node (Q1) at (0,0) {$\mathbb{Q}$};
			\node (Q2) at (3,2) {$\mathbb{Q}(\zeta_8+\zeta_8^{-1})$};
			\node (Q3) at (0,2) {$\mathbb{Q}(\zeta_8-\zeta_8^{-1})$};
			\node (Q4) at (-3,2) {$\mathbb{Q}(\zeta_4)$};
			\node (Q5) at (-4,4) {$\mathbb{Q}(\zeta_8)$};
			\node (Q6) at (0,4) {$\mathbb{Q}(\zeta_{16}-\zeta_{16}^{-1})$};
			\node (Q7) at (4,4) {$\mathbb{Q}(\zeta_{16}+\zeta_{16}^{-1})$};
			\node (Q8) at (0,8) {$\mathbb{Q}(\zeta_{16})$};
			\draw (Q1)--(Q2) node [pos=0.7, below,inner sep=0.25cm] {2};
			\draw (Q1)--(Q4) node [pos=0.7, below,inner sep=0.25cm] {2};
			\draw (Q1)--(Q3) node [pos=0.7, right,inner sep=0.16cm] {2};
			\draw (Q4)--(Q5) node [pos=0.7, below,inner sep=0.25cm] {2};
			\draw (Q3)--(Q5) node [pos=0.7, below,inner sep=0.25cm] {2};
			\draw (Q2)--(Q5) node [pos=0.4, above,inner sep=0.16cm] {2};
			\draw (Q2)--(Q6) node [pos=0.4, above,inner sep=0.25cm] {2};
			\draw (Q2)--(Q7) node [pos=0.7, below,inner sep=0.25cm] {2};
			\draw (Q6)--(Q8) node [pos=0.5, right,inner sep=0.16cm] {2};
			\draw (Q7)--(Q8) node [pos=0.7, below,inner sep=0.25cm] {2};
			\draw (Q5)--(Q8) node [pos=0.7, below,inner sep=0.25cm] {2};
		\end{tikzpicture}}
	\end{center} 	
\noindent When the base field is extended from \(\mathbb{Q}\) to \(\mathbb{Q}(\zeta_4)\), the 2-power cyclotomic extension \(\mathbb{Q}(\zeta_4)(\zeta_{16})/\mathbb{Q}(\zeta_4)\) becomes cyclic. Specifically, we obtain a unique tower decomposition of \(\mathbb{Q}(\zeta_4)(\zeta_{16})/\mathbb{Q}(\zeta_4)\) as follows:
	\begin{center}
			\scalebox{0.7}{\begin{tikzpicture}
			\node (Q1) at (0,0) {$\mathbb{Q}(\zeta_4)$};
			\node (Q2) at (0,2) {$\mathbb{Q}(\zeta_4)(\zeta_8+\zeta_8^{-1})=\mathbb{Q}(\zeta_4)(\zeta_8)=\mathbb{Q}(\zeta_4)(\zeta_8-\zeta_8^{-1})$};
			\node (Q3) at (0,4) {$\mathbb{Q}(\zeta_4)(\zeta_{16}+\zeta_{16}^{-1})=\mathbb{Q}(\zeta_4)(\zeta_{16})=\mathbb{Q}(\zeta_4)(\zeta_{16}-\zeta_{16}^{-1})$};
			\draw (Q1)--(Q2) node [pos=0.7,right,inner sep=0.25cm] {2};
			\draw (Q2)--(Q3) node [pos=0.7,right,inner sep=0.25cm] {2};
		\end{tikzpicture}}
	\end{center}
	
\noindent Another example is over the field of positive characteristic \(\mathbb{F}_3\). It is well known that \(\mathbb{F}_3(\zeta_{16})/\mathbb{F}_3\) is a cyclic extension of degree 4 (see \cite[Theorem 4.1]{conradcyclotomic}). Moreover, over \(\mathbb{F}_3\), the cyclotomic polynomial \(\Phi_{16}(x)\) splits into the irreducible polynomials \(x^4 - x^2 - 1\) and \(x^4 + x^2 - 1\). Thus, \(\zeta_{16}\) is a root of either \(x^4 - x^2 - 1\) or \(x^4 + x^2 - 1\), which implies \(\zeta_8\) is a root of \(x^2 + x - 1\) or \(x^2 - x - 1\). In both cases, for a non-trivial element \(\sigma \in \operatorname{Gal}(\mathbb{F}_3(\zeta_{16})/\mathbb{F}_3)\), \(\sigma(\zeta_8) = -\zeta_8^{-1}\). Thus, \(\zeta_8 - \zeta_8^{-1} \in \mathbb{F}_3\). From this, we obtain the following unique tower decomposition of \(\mathbb{F}_3(\zeta_{16})/\mathbb{F}_3\):
		\begin{center}
		\scalebox{0.7}{\begin{tikzpicture}
			\node (Q1) at (0,0) {$\mathbb{F}_3=\mathbb{F}_3(\zeta_8-\zeta_8^{-1})$};
			\node (Q2) at (0,2) {$\mathbb{F}_3(\zeta_8)=\mathbb{F}_3(\zeta_8+\zeta_8^{-1})$};
			\node (Q3) at (0,4) {$\mathbb{F}_3(\zeta_{16}+\zeta_{16}^{-1})=\mathbb{F}_3(\zeta_{16})=\mathbb{F}_3(\zeta_{16}-\zeta_{16}^{-1})=\mathbb{F}_{3^4}$};
			\draw (Q1)--(Q2) node [pos=0.7,right,inner sep=0.25cm] {2};
			\draw (Q2)--(Q3) node [pos=0.7,right, inner sep=0.25cm] {2};
		\end{tikzpicture}}
	\end{center}
Through these examples, we observe different properties that affect the cyclicity of 2-power cyclotomic extensions. In this paper, we provide various characterizations for when a 2-power cyclotomic extension is cyclic.

We consistently assume that \([F(\zeta_{2^e}):F] > 2\). This is because the cyclicity property always holds when \([F(\zeta_{2^e}):F] \leq 2\), and a thorough examination of the quadratic case has been conducted in \cite{marques-quadratic}.

In \cite{marques-quadratic}, we demonstrated that the constant \(\nu_{2^\infty_F}\), as defined in Definition \ref{def-tfn}, is essential for describing quadratic cyclotomic extensions. This constant is also crucial in determining the cyclicity of 2-power cyclotomic extensions. Our main contribution in this paper is to characterize the cyclicity of 2-power cyclotomic extensions based solely on a condition of the base field (see Theorem \ref{maincyclic} and Theorem \ref{noncyclic-theorem}). Specifically, we prove that:

\[
F(\zeta_{2^e})/F \text{ is cyclic if and only if } \zeta_4 \in F \text{ or } \zeta_4 \notin F \text{ and } \zeta_{2^{\nu_{2^\infty_F}}} - \zeta_{2^{\nu_{2^\infty_F}}}^{-1} \in F.
\]
A characterization based on the base field was expected due to the nature of cyclotomic extensions. However, our literature review did not find any existing characterization of this kind. Furthermore, the condition on the base field that ensures cyclicity can be directly verified by computer for finite fields, for instance.

The constants \(\nu_{2^\infty_F}\) and \(\zeta_{2^{\nu_{2^\infty_F}}} - \zeta_{2^{\nu_{2^\infty_F}}}^{-1}\) play a crucial role in Schinzel's theorem, which classifies radical extensions up to isomorphism (see \cite[Theorem 3]{schinzel}, \cite[Theorem 2.1]{several}). It is well-known that simple radical extensions are Kummer if and only if they have a unique isomorphism type: two simple radical extensions over a field \( F \) of degree \( n \), defined by the equations \( x^n - a \) and \( x^n - b \), respectively, are isomorphic if and only if \( a = c^n b^j \) for some \( c \in F \) and an integer \( j \) coprime with \( n \). 

Schinzel demonstrated the emergence of another type of relation in the non-Kummer case, specifically involving the constants \(\nu_{2^\infty_F}\) and \(\zeta_{2^{\nu_{2^\infty_F}}} - \zeta_{2^{\nu_{2^\infty_F}}}^{-1}\). However, his work and subsequent studies utilizing his theorem such as \cite[Theorem 4.2]{subfieldsradicalextension}, \cite[Theorem 2.1]{several}, did not clarify precisely when this new type of relation arises or whether it can be determined directly from a condition on the base field. 

More recent work \cite{zhang2017grunwald} attempted to address this question, obtaining a partial answer in characteristic 0 \cite[Theorem 3.1.4]{zhang2017grunwald}, and mentioning in a subsequent remark the expected difficulty of obtaining a full characterization even in characteristic 0. 

With the present characterization of cyclicity, we can now prove that this new isomorphism type appears precisely when those radical extensions contain a non-cyclic 2-power cyclotomic extension. Furthermore, we can characterize this property with a condition on the base field and on the coefficient of the radical polynomial \cite[Theorem 5.2.6]{thesis}.

The paper is structured as follows: We begin by introducing notation and presenting elementary results. In the following section, we state and prove our first main theorem, characterizing the cyclicity of 2-power cyclotomic extensions by means of the structure of the Galois groups, the nature of their subextensions, tower decompositions, and, most importantly, specific conditions on the base field (Theorem \ref{maincyclic}). In the third section, we study the properties of the subextensions of 2-power cyclotomic extensions \( F(\zeta_{2^e}) /F \) generated by \(\zeta_{2^e} \pm \zeta_{2^e}^{-1}\) (see Theorem \ref{te}). We conclude this paper by characterizing the non-cyclicity of 2-power cyclotomic extensions, which is not simply the negation of the cyclic case (see Theorem \ref{noncyclic-theorem}).

\newpage
	
	\section*{Notation and symbols}
	In this paper, $F$ denote  a field with characteristic $\wp$, $\overline{F}$ is one fixed algebraic closure of the field $F$, and $e$ is a positive integer. 
	
	To maintain simplicity in this paper, we will assume that all discussed field extensions are subfields of the chosen algebraic closure $\overline{F}$.
	In the following table, we present notation and symbols that will be utilized throughout this paper.
	\begin{center} 
		\begin{tabular}{l p{11cm} }
			$\mathbb{N}$ & The set of natural number including zero.\\
			$\llbracket  m, n  \rrbracket$ & $\{ m, m+1,\cdots , n \}$ where $m, n \in \mathbb{N}$ such that $m \leq n$.\\
			$U_{n}$ & The subset of $\mathbb{Z}/n\mathbb{Z}$ consisting of invertible elements under multiplication where $n\in \mathbb{N}$.\\
			$[j]_{n}$ & The equivalence class of $j\in \mathbb{Z}$ modulo $n$. We shall simply use the notation $[j]$ when $n$ is clear from the context.\\
			$\sigma_k$ & The $F$-homomorphism from $F(\zeta_n)$ to $F(\zeta_n)$ sending $\zeta_n$ to $\zeta_n^k$, where $k\in \nn{n-1}$.\\
			$[\sigma]_H$ & The equivalence class of $\sigma\in G$ in the quotient group $G\big /H$ where $G$ is a galois group. We shall simply use the notation $[\sigma]$ when $H$ is clear from the context.\\
			${\langle [a] \rangle}_G$ & A subgroup generated by $[a]$ where $G$ is either $\mathbb{Z}/n\mathbb{Z}$ or $U_n$ with $n\in \mathbb{N}$. \\
			$o_F(\alpha)$ & The order of an element $\alpha\in \overline{F}$ over a field $F$.\\
			$\operatorname{min}(\alpha, F)$ & The minimal polynomial of an algebraic element $\alpha\in \overline{F}$ over the field $F$.\\
			$\tr{2^e}$ & $\zeta_{2^e} \pm \zeta_{2^e}^{-1}$.\\
			
$\operatorname{t}_F(2^e)$ & $\left\{ \begin{array}{clll} 
	2^e & \text{when } \operatorname{o}_F(\zeta_{2^e})>2;\\
	2 & \text{when }  \operatorname{o}_F(\zeta_{2^e})=2;\\
	1 & \text{when } \operatorname{o}_F(\zeta_{2^e})=1.
\end{array} \right.$ \\
			$\nu_{2^\infty_F}^{+}$ & $\left\{ \begin{array}{clll}\operatorname{max}\{k \in \mathbb{N} | \zeta_{t_F(2^k)} +
				\zeta_{t_F(2^k)}^{-1} \in F\} & \text{ when it exists }; \\
				\infty & \text{ otherwise.} \end{array} \right.$\\
			$\nu_{2^\infty_F}$ & $\left\{ \begin{array}{clll} \nu_{2^\infty_F}^++1 & \text{when}\ F \ \text{has property } \ \mathcal{C}_2 \ (\text{see \cite[Definition 3.13]{marques-quadratic}}) ; \\ 
				\nu_{2^\infty_F}^+ & \text{ otherwise. }  \end{array} \right.$ 
		\end{tabular}
	\end{center}
	\newpage
	\section{Definitions and elementary results}
		We begin by presenting the following fundamental definitions that establish the notational conventions utilized throughout this paper. We start by introducing  the concept of  field subextensions and then proceed to define the $2$-tower decomposition of  a  field extension. 
	\begin{definition}
		Given two extensions $K/F$ and $M/F$. We say that $M/F$ is {\sf a subextension of $K/F$} if $F \subseteq M \subseteq K$. Furthermore, we say that $M/F$ is {\sf a subextension of $K/F$ of codegree $n$} if $M/F$ is an subextension of $K/F$ and $[K:M]=n.$
	\end{definition} 
	
	\begin{definition}
		Given an extension $K/F$ of degree $2^e$. We call a {\sf $2$-tower decomposition} of $K/F$ denoted by $K_s/K_{s-1}/\cdots/K_1/K_0$ where $K_0=F$, $K_s=K$, $K_{i}/F$ is an intermediate subextension of $K_{i+1}/F$ of codegree $2$ for all $i\in \llbracket0, s-1\rrbracket$.  
	\end{definition} 
	
	In the following definition, we introduce essential constants and a function that will be used throughout this paper (see \cite[Definition 2.3.1]{thesis}, \cite[Definition 2.3.15]{thesis} and \cite[Definition 2.3.18]{thesis}). 
\begin{definition}\label{def-tfn}
	
	\begin{enumerate}
		\item We define  $$\operatorname{t}_F(2^e) =\left\{ \begin{array}{clll} 
2^e & \text{when } \operatorname{o}_F(\zeta_{2^e})>2;\\
2 & \text{when }  \operatorname{o}_F(\zeta_{2^e})=2;\\
1 & \text{when } \operatorname{o}_F(\zeta_{2^e})=1.
\end{array} \right.$$ 
		
	\item We define 	$$\nu_{2^\infty_F}^+=\left\{ \begin{array}{clll} \max \{ k \in \mathbb{N} | \zeta_{t_F(2^k)} +
		\zeta_{t_F(2^k)}^{-1} \in F \} & \text{ when it exists} \\ 
		\infty & \text{ otherwise.}  \end{array} \right.$$
		
		\item  We define 	$$\nu_{2^\infty_F}=\left\{ \begin{array}{clll} \nu_{2^\infty_F}^++1 & \text{when} \ p=2\ \text{ and $F$ has property $\mathcal{C}_2$  (see \cite[Definition 3.13]{marques-quadratic})};  \\ 
			\nu_{2^\infty_F}^+ & \text{ otherwise. }  \end{array} \right.$$ 
		\end{enumerate}
\end{definition}
\begin{remark}
	
	 Under the assumption $[F(\zeta_{2^e}):F]
>2$,  we have $o_F(\zeta_{2^e})>2$ so that  by Definition  \ref{def-tfn}  above $t_F(2^e)=2^e$. Therefore, we obtain $$\nu_{2^\infty_F}^+=\left\{ \begin{array}{clll}\operatorname{max}\{k \in \mathbb{N} |\zeta_{2^k}+\zeta_{2^k}^{-1}\in F\} & \text{ when it exists }; \\
	\infty & \text{ otherwise.} \end{array} \right.$$
\end{remark}
In the remaining part of the paper, we will assume that $\nu_{2^\infty_F}<\infty$ and we will often make the assumption that $e>\nu_{2^\infty_F}$. This assumption is equivalent to considering the field extension $F(\zeta_{2^e})/F$ has a degree at least $4$, as proven in the following Lemma.
\begin{lemma}\label{lem}
Let $\nu_{2^\infty_F}<\infty$. Then $e>\nu_{2^\infty_F}$ if and only if $[F(\zeta_{2^{e}}):F]\geq 4$. In particular, $\nu_{2^\infty_F}+1$ is the smallest integer $e$ such that $[F(\zeta_{2^{e}}):F]= 4$. 
\end{lemma}
\begin{proof}
Suppose $e>\nu_{2^\infty_F}$. By \cite[Theorem 3.24]{marques-quadratic}, we have $[F(\zeta_{2^{\nu_{2^\infty_F}}}):F]=2$. Moreover, $[F(\zeta_{2^{\nu_{2^\infty_F}+1}}):F(\zeta_{2^{\nu_{2^\infty_F}}})]=2$ since $\zeta_{2^{\nu_{2^\infty_F}+1}}\notin F(\zeta_{2^{\nu_{2^\infty_F}}})$, as otherwise either $\zeta_{2^{\nu_{2^\infty_F}}}\in F$ or $\tq{2^{\nu_{2^\infty_F}+1}}\in F$, contradicting the definition of $\nu_{2^\infty_F}$. Additionally, $\zeta_{2^{\nu_{2^\infty_F}+1}}$ is a root of the polynomial $x^2-\zeta_{2^{\nu_{2^\infty_F}}}$ over $F(\zeta_{2^{\nu_{2^\infty_F}}})$. Therefore, we get $[F(\zeta_{2^{\nu_{2^\infty_F}+1}}):F]=4$ by the multiplicativity of the degree of a tower of extensions. As a result, $[F(\zeta_{2^{e}}):F]\geq 4$ since $e\geq \nu_{2^\infty_F}+1$. The converse can be obtained easily with a simple contrapositive. 
\end{proof}

\section{Cyclic \texorpdfstring{$2$}{Lg}-power cyclotomic extensions }

In this section, we characterize when the extension \( F(\zeta_{2^e})/F \) is cyclic and when it is not, by establishing necessary and sufficient conditions on the base field \( F \). Additionally, we describe its subextensions of codegree $2$, all its 2-tower decompositions, and the degrees of these extensions based on the cyclicity property. We present the main theorem for a cyclic 2-power cyclotomic extension in the following result.

		\begin{theorem}\label{maincyclic}
		Suppose that $e>\nu_{2^\infty_F}$. 
			The following assertions are equivalent:
			\begin{enumerate}
				\item $\operatorname{Gal}(F(\zeta_{2^e})/F)$ is a cyclic group.
				\item $\operatorname{Gal}(F(\zeta_{2^e})/F)=\{\sigma_{j}|[j]\in H\}$  where either  $H\leq {\langle [5]\rangle}_{U_{2^e}}$ or $H\leq {\langle [(-5)]\rangle}_{U_{2^e}}.$
				\item $F(\zeta_{2^e})/F$ admits a unique codegree $2$ subextension $F(\zeta_{2^{e-1}})/F$.
				\item $F(\zeta_{2^e})=F(\ts{2^e})=F(\tq{2^e})$.
				\item $F(\zeta_{2^e})/F$ admits a unique $2$-tower decomposition. 
				\begin{center} $F(\zeta_{2^e})/F(\zeta_{2^{e-1}})/\cdots/F(\zeta_{\scalebox{0.7}{$2^{\nu_{2^\infty_F}}$}})/F.$ \end{center} 
			 	\item $\zeta_4\in F$ or $\zeta_4\notin F$ and $\tq{2^{\scalebox{0.7}{$\! \nu_{\scalebox{0.7}{$2^\infty_F$}}$}}} \in F$.
			\end{enumerate}
				Moreover, under the cyclicity assumption, we also $F(\zeta_{2^{\scalebox{0.7}{$\nu_{2^\infty_F}$}}})= F(\ts{2^{\scalebox{0.5}{$\nu_{2^\infty_F}$}}})$.
		\end{theorem}
		
	\begin{proof}
	$(1) \Longleftrightarrow (2)$	We suppose $\operatorname{Gal}(F(\zeta_{2^e})/F)$ is cyclic. Since by Lemma \ref{lem} $[F(\zeta_{2^e}):F]>2$, and by leveraging \cite[Theorem 2.3]{conradcyclotomic} along with \cite[Lemma 2.1]{conradcyclotomic}, we establish that $\operatorname{Gal}(F(\zeta_{2^e})/F)$ is of order strictly greater than $2$, and isomorphic to a cyclic subgroup of $U_{2^e}$. Thus, the result follows. The converse is trivial.  \\

		$(2) \Longleftrightarrow (3)$  We suppose \(\operatorname{Gal}(F(\zeta_{2^e})/F) = \{\sigma_{j} \mid [j] \in H\}\) where \(H \leq \langle [5] \rangle_{U_{2^e}}\). Similar arguments may be used to prove the case when \(H \leq \langle [-5] \rangle_{U_{2^e}}\).

Since \(H \leq \langle [5] \rangle_{U_{2^e}}\), \([5^{2^{e-3}}]\) is the unique element of order 2 in \(U_{2^e}\), as \(5^{2^{e-3}} \equiv 2^{e-1} + 1 \mod 2^e\) and \(H\) is cyclic. Consequently, \(\sigma_{5^{2^{e-3}}}(\zeta_{2^e}) = \zeta_{2^e}^{2^{e-1} + 1} = -\zeta_{2^e}\) and \(x^2 - \zeta_{2^{e-1}}\) is the minimal polynomial of \(\zeta_{2^e}\) over \(F(\zeta_{2^{e-1}})\). Since \([F(\zeta_{2^e}) : F(\zeta_{2^e})^{\langle \sigma_{5^{2^{e-3}}} \rangle}] = 2\), this implies \(F(\zeta_{2^e})^{\langle \sigma_{5^{2^{e-3}}} \rangle} = F(\zeta_{2^{e-1}})\), as \(F(\zeta_{2^e})/F\) is a cyclic extension. The converse follows by using contradiction and the Galois correspondence.\\

		$(3) \Longrightarrow (4)$ 
			  Follows from
		\cite[Lemma 2.9]{marques-quadratic}, that $\zeta_{2^e}$ is a root of the polynomial either $x^2-\zeta_{2^{e-1}}$, or $x^2-\ts{2^e}x+1$ or $x^2-\tq{2^e}x-1$ and the uniqueness of the minimal polynomial of $\zeta_{2^e}$ over a subextension of codegree $2$.  \\
		
		$(4) \Longrightarrow (3)$  From
		\cite[Lemma 2.9]{marques-quadratic} we know that the minimal polynomial of $\zeta_{2^e}$ over the codegree $2$ extension which exists and is unique by the Galois correspondence is either $x^2-\zeta_{2^{e-1}}$, or $x^2-\ts{2^e}x+1$ or $x^2-\tq{2^e}x-1$, but since $F(\zeta_{2^e})=F(\ts{2^e})=F(\tq{2^e})$ the only possibility is that $x^2-\zeta_{2^{e-1}}$ is the minimal polynomial and $(3)$ is satisfied.\\

		$(3) \Longleftrightarrow (5)$  We assume that  $F(\zeta_{2^e})/F$ admits a unique codegree $2$ subextension $F(\zeta_{2^{e-1}})/F$. We want to prove that $F(\zeta_{2^e})/F$ admits a unique $2$-tower decomposition $$F(\zeta_{2^e})/F(\zeta_{2^{e-1}})/\cdots/F(\zeta_{\scalebox{0.7}{$2^{\nu_{2^\infty_F}}$}})/F.$$
		We will prove the statement by induction on the degree $[F(\zeta_{2^e}):F]$. We start by the base case $[F(\zeta_{2^e}):F]=4$.  Since $[F(\zeta_{2^e}):F(\zeta_{2^{e-1}})]=2$ and $[F(\zeta_{2^e}):F]=4$, then by multiplicativity of the degree of a tower of field extension we have $[F(\zeta_{2^{e-1}}):F]=2$. Moreover, since $F(\zeta_{2^{e-1}})/F$ is unique codegree $2$ subextension and by Lemma \ref{lem} $e=\nu_{2^\infty_F}+1$, then $F(\zeta_{2^e})/F(\zeta_{\scalebox{0.7}{$2^{\nu_{2^\infty_F}}$}})/F$  is the unique $2$-tower decomposition of $F(\zeta_{2^e})/F$ as desired.
The rest of the induction is not hard to derive. Moreover, the converse is trivial. 	\\
			
	 	$ (5) \Longrightarrow (6)$ 		 Suppose that $F(\zeta_{2^e})/F(\zeta_{2^{e-1}})/\cdots/F(\zeta_{2^{\scalebox{0.7}{$\nu_{2^\infty_F}$}}})/F$ is the unique $2$-tower \\ decomposition of $F(\zeta_{2^e})/F$. As established in the proof of $(3) \Longleftrightarrow (4)  \Longleftrightarrow (5)$ and $[F(\zeta_{2^{\scalebox{0.5}{$\nu_{2^\infty_F}+1$}}}):F]=4$ by Lemma \ref{lem}, we have $F(\zeta_{2^{\nu_{2^\infty_F}+1}})/F(\zeta_{2^{\scalebox{0.7}{$\nu_{2^\infty_F}$}}})/F$ as the unique $2$-tower decomposition of the extension $F(\zeta_{2^{\scalebox{0.7}{$\nu_{2^\infty_F}+1$}}})/F$ and $F(\zeta_{2^{\scalebox{0.5}{$\nu_{2^\infty_F}+1$}}})=F(\ts{2^{\scalebox{0.5}{$\nu_{2^\infty_F}+1$}}})=F(\tq{2^{\scalebox{0.7}{$\nu_{2^\infty_F}+1$}}})$. 
		Moreover, we obtain that $[F(\ts{2^{\scalebox{0.5}{$\nu_{2^\infty_F}+1$}}}): F(\ts{2^{\scalebox{0.5}{$\nu_{2^\infty_F}$}}})]=2$, indeed $\ts{2^{\scalebox{0.5}{$\nu_{2^\infty_F}+1$}}}$ satisfy the polynomial $x^2-(\ts{2^{\scalebox{0.5}{$\nu_{2^\infty_F}$}}}+2)$ over $F(\ts{2^{\nu_{2^\infty_F}}})$ and the Galois group of $F(\zeta_{2^{\nu_{2^\infty_F}+1}})/F(\zeta_{2^{\scalebox{0.7}{$\nu_{2^\infty_F}$}}})$ stabilize $\ts{2^{\scalebox{0.5}{$\nu_{2^\infty_F}$}}}$.
		Therefore, by the uniqueness of $2$-tower decomposition, we obtain $F(\zeta_{2^{\scalebox{0.7}{$\nu_{2^\infty_F}$}}})= F(\ts{2^{\scalebox{0.5}{$\nu_{2^\infty_F}$}}})$. It then follows by  \cite[Lemma 2.9]{marques-quadratic} that, $\operatorname{min}(\zeta_{2^{\nu_{2^\infty_F}}}, F)$ is  either $x^2-\zeta_{2^{\nu_{2^\infty_F}}}^2$ or $x^2-(\tq{2^{\nu_{2^\infty_F}}})x-1$.  By \cite[Theorem 3.24]{marques-quadratic}, the first case implies that $\zeta_4 \in F$ and the second case implies that $\zeta_4\notin F$  and  $\tq{2^{\scalebox{0.7}{$\nu_{2^\infty_F}$}}} \in F$. \\

		$(6) \Longrightarrow (1)$  Suppose that $\zeta_4\in F$ or $\zeta_4\notin F$ and $\tq{2^{\scalebox{0.5}{$\nu_{2^\infty_F}$}}} \in F$. 
		
	Suppose that $\zeta_4 \in F$. Consider a non-trivial automorphism $\sigma$ in $\operatorname{Gal}(F(\zeta_{2^e})/F)$. By \cite[Lemma 2.1]{conradcyclotomic}, $\sigma(\zeta_{2^e})=\zeta_{2^e}^{j_\sigma}$ for some $[j_\sigma] \in U_{2^e}$. Since $\zeta_4 = \zeta_{2^e}^{2^{e-2}}\in F$, that is equivalent to have $\zeta_{2^e}^{2^{e-2}j_\sigma} = \zeta_{2^e}^{2^{e-2}}$ which is equivalent to have $j_\sigma \equiv 1 \mod 4$. Additionally, \cite[Theorem 2.3]{conradcyclotomic} shows that $\operatorname{Gal}(F(\zeta_{2^e})/F)$ is isomorphic to a subgroup of $U_{2^e}$. We recall that $U_{2^e}=\{ [\epsilon 5^k] | k \in \llbracket 0, 2^{e-2}-1 \rrbracket, \epsilon \in { \pm 1 } \}$. Thus,  $\epsilon5^k\equiv 1 \mod 4$ is equivalent to $\epsilon=1$, and that is in turn equivalent to have $\operatorname{Gal}(F(\zeta_{2^e})/F)$ is isomorphic to a subgroup of the cyclic group generated by $[5]$. Hence, the Galois group is cyclic.

		It remains to prove that  the condition   $\zeta_4\notin F$ and $\tq{2^{\nu_{2^\infty_F}}} \in F$ implies   $\operatorname{Gal}(F(\zeta_{2^e})/F)$ is a cyclic group, equivalently from the above, that $F(\zeta_{2^e})/F$ admits a unique codegree $2$ subextension $F(\zeta_{2^{e-1}})/F$.
		 We proceed by induction on the degree of $F(\zeta_{2^e})/F$. We start with the base case,  $[F(\zeta_{2^e}):F]=4.$
		 
		 By \cite[Theorem 3.24]{marques-quadratic} we know that $ [F(\zeta_{2^{\scalebox{0.5}{$\nu_{2^\infty_F}$}}}):F]=2$. Also, since $[F(\zeta_{2^e}):F]=4$ then  $e=\nu_{2^\infty_F}+1$ by Lemma \ref{lem}. Thus, by the multiplicativity property of degree of a tower of field extensions we have $[F(\zeta_{2^{\scalebox{0.5}{$\nu_{2^\infty_F}+1$}}}):F(\zeta_{2^{\scalebox{0.5}{$\nu_{2^\infty_F}$}}})]=2$. It is enough to prove that $F(\zeta_{2^{\scalebox{0.5}{$\nu_{2^\infty_F}$}}})/F$ is the only subextension of $F(\zeta_{2^{\scalebox{0.5}{$\nu_{2^\infty_F}+1$}}})/F$ of codegree $2$. We argue by contradiction. By \cite[Lemma 2.9]{marques-quadratic} we know that the possible codegree $2$ subextension of $F(\zeta_{2^{\scalebox{0.5}{$\nu_{2^\infty_F}+1$}}})/F$ are  $F(\zeta_{2^{\scalebox{0.5}{$\nu_{2^\infty_F}$}}})/F$, $F(\ts{2^{\scalebox{0.5}{$\nu_{2^\infty_F}+1$}}})/F$ and $F(\tq{2^{\scalebox{0.5}{$\nu_{2^\infty_F}+1$}}})/F$.  Now we suppose that $F(\zeta_{2^{\scalebox{0.5}{$\nu_{2^\infty_F}$}}})/F$ and    $F(\ts{2^{\scalebox{0.5}{$\nu_{2^\infty_F}+1$}}})/F$ are distinct subextensions of codegree $2$ of $F(\zeta_{2^{\scalebox{0.5}{$\nu_{2^\infty_F}+1$}}})/F$.
		  Since $\ts{2^{\scalebox{0.5}{$\nu_{2^\infty_F}+1$}}}$ is a root of the polynomial $x^2-(\ts{2^{\scalebox{0.5}{$\nu_{2^\infty_F}$}}}+2)$ and  $[F(\ts{2^{\scalebox{0.5}{$\nu_{2^\infty_F}+1$}}}):F]=2$ by the multiplicativity property of degree of a tower of field extensions, then the polynomial $x^2-(\ts{2^{\scalebox{0.5}{$\nu_{2^\infty_F}$}}}+2)$ is irreducible over $F$ and $\ts{2^{\scalebox{0.5}{$\nu_{2^\infty_F}$}}}\in F$. But that is impossible since the non-trivial element in $\operatorname{Gal}(F(\zeta_{2^{\scalebox{0.5}{$\nu_{2^\infty_F}$}}})/F)$ sends $\ts{2^{\scalebox{0.5}{$\nu_{2^\infty_F}$}}}$ to $-\ts{2^{\scalebox{0.5}{$\nu_{2^\infty_F}$}}}$. 
		   As a consequence,  we have proven that $F(\ts{2^{\scalebox{0.5}{$\nu_{2^\infty_F}+1$}}})/F$ is not a subextension of $F(\zeta_{2^{\scalebox{0.5}{$\nu_{2^\infty_F}+1$}}})/F$ with codegree $2$. Using a similar method, we can prove that  $F(\tq{2^{\scalebox{0.5}{$\nu_{2^\infty_F}+1$}}})/F$ is not admissible as a codegree $2$ subextension. 
		 Therefore, $F(\zeta_{2^{\scalebox{0.5}{$\nu_{2^\infty_F}$}}})/F$ is the only subextension of codegree $2$ of $F(\zeta_{2^{\scalebox{0.5}{$\nu_{2^\infty_F}+1$}}})/F$. 
		  
		  Now suppose that the property is true for any $2$-power cyclotomic extension of degree $2^m, m>2$. We will prove that this property remains true for a $2$-power cyclotomic extension $F(\zeta_{2^e})/F$ of degree $2^{m+1}$.
		  
		  By \cite[Lemma 2.9]{marques-quadratic} we know that the possible subextensions of codegree $2$ of $F(\zeta_{2^e})/F$ are $F(\zeta_{2^{e-1}})/F$, $F(\ts{2^e})/F$ and $F(\tq{2^e})/F$. 
		  We will prove that $F(\zeta_{2^{e-1}})/F$ is the only admissible subextension of codegree $2$ of $F(\zeta_{2^{e}})/F$. 
		  Since $m\geq 2$, we know by the induction assumption that $F(\zeta_{2^{e-1}})/F$ admits a unique $2$-tower decomposition and by the equivalences established above, we know that $F(\zeta_{2^{e-1}})/F$ is a cyclic extension and  $F(\zeta_{2^{e-1}})= F(\ts{2^{e-1}})=F(\tq{2^{e-1}})$.  We now argue by contradiction. 
	   Suppose that $F(\zeta_{2^{e-1}})/F$ and $F(\ts{2^{e}})/F$ are distinct codegree $2$ subextensions of $F(\zeta_{2^{e}})/F$. The other cases can be proven similarly. On the one hand, from our contradiction assumption, $[F(\zeta_{2^{e}}):F(\ts{2^{e}})]=2$. On the other hand, from the above we have $F(\zeta_{2^{e-1}})= F(\ts{2^{e-1}})=F(\tq{2^{e-1}})$. Thus,  $[F(\ts{2^{e}}):F(\ts{2^{e-1}})]= 2$,  since $\ts{2^{e}}$ is a root of the polynomial $x^2 - (\ts{2^{e-1}}+2)$ and our contradiction assumption. But then by multiplicativity of the degree in towers, we have $[F(\zeta_{2^{e}}):F(\ts{2^{e-1}})]=4$ contradicting that this degree should be $2$ given that $F(\zeta_{2^{e-1}})= F(\ts{2^{e-1}})=F(\tq{2^{e-1}})$.
		Therefore, we can conclude the induction. The remaining part of the proof is obtained in the arguments above.

	\end{proof}

	Based on the proof of Theorem \ref{maincyclic} above, we deduce the following corollary. 
	\begin{corollary}
		Suppose that $e>\nu_{2^\infty_F}$, then
		\begin{itemize}
			\item  $\operatorname{Gal}(F(\zeta_{2^e})/F)$ is isomorphic to a subgroup of the cyclic group  $ {\langle [5]\rangle}_{U_{2^e}}$ if and only if  $\zeta_4\in F$.
			\item $\operatorname{Gal}(F(\zeta_{2^e})/F)$ is isomorphic to a subgroup of the cyclic group $ {\langle [-5]\rangle}_{U_{2^e}}$ if and only if $\zeta_4\notin F$ and $\tq{2^{\scalebox{0.7}{$\! \nu_{\scalebox{0.7}{$2^\infty_F$}}$}}} \in F$.
		\end{itemize} 
	\end{corollary}
	We can also deduce from Theorem \ref{maincyclic} an expression of the minimal polynomial when the 2-power cyclotomic extension is cyclic.
 \begin{corollary}
 	Let $e>\nu_{2^\infty_F}$. If $\operatorname{Gal}(F(\zeta_{2^e})/F)$ is a cyclic group, then $[F(\zeta_{2^e}):F]=2^{\scalebox{0.7}{$e-(\nu_{2^\infty_F}-1)$}}$ and $\operatorname{min}(\zeta_{2^{e-i}}, F(\zeta_{2^{e-(i+k)}}))=x^{2^k}-\zeta_{2^{e-(i+k)}}$ for all $i\in \llbracket0, e-\nu_{2^\infty_F}\rrbracket$ and $k\in \llbracket 0, e-\nu_{2^\infty_F}-i \rrbracket.$ Moreover, 
			\begin{itemize}
				\item when $\zeta_4\in F$,  $\operatorname{min}(\zeta_{2^{e-i}}, F)=x^{2^{\scalebox{0.7}{$e-(\nu_{2^\infty_F}+i-1)$}}}-\zeta_{\scalebox{0.5}{$2^{\nu_{2^\infty_F}-1}$}}$ for all $i\in \llbracket0, e-(\nu_{2^\infty_F}-1)\rrbracket$. In particular, $\operatorname{min}(\zeta_{2^{e}}, F)=x^{2^{\scalebox{0.5}{$e-(\nu_{2^\infty_F}-1)$}}}-\zeta_{\scalebox{0.5}{$2^{\nu_{2^\infty_F}-1}$}}$, and
				\item when $\zeta_4\notin F$ and  $\tq{2^{\scalebox{0.5}{$\nu_{2^\infty_F}$}}} \in F$, $\operatorname{min}(\zeta_{2^{e-i}}, F)=x^{2^{\scalebox{0.7}{$e-(\nu_{2^\infty_F}+i-1)$}}}-\tq{2^{\scalebox{0.5}{$\nu_{2^\infty_F}$}}}x^{2^{\scalebox{0.7}{$e-(\nu_{2^\infty_F}+i)$}}}-1$ for all $i\in \llbracket0, e-\nu_{2^\infty_F}\rrbracket$. In particular, $\operatorname{min}(\zeta_{2^{e}}, F)=x^{2^{\scalebox{0.7}{$e-(\nu_{2^\infty_F}-1)$}}}-\tq{2^{\scalebox{0.5}{$\nu_{2^\infty_F}$}}}x^{2^{\scalebox{0.7}{$e-\nu_{2^\infty_F}$}}}-1$.
			\end{itemize}
 \end{corollary}
		
		\begin{proof} \text{ \ } 
	\begin{itemize} 		
	\item Suppose that $\zeta_4 \in F$. Let $i \in \llbracket 0, e-\nu_{2^\infty_F} \rrbracket$ and $k \in \llbracket 1, e-\nu_{2^\infty_F}-i \rrbracket$. 
	Clearly, $\zeta_{2^{e-i}}$ satisfies the polynomial $x^{2^k} - \zeta_{2^{e-(i+k)}}$. 
	Furthermore, as established in Theorem \ref{maincyclic} (3), we have for all $i \in \llbracket 0, e-\nu_{2^\infty_F} \rrbracket$,  $[F(\zeta_{2^{e-i}}):F(\zeta_{2^{e-(i+1)}})] = 2$. Using the multiplicativity property of the degree of a tower of field extensions, it follows that $[F(\zeta_{2^{e-i}}): F(\zeta_{2^{e-(i+k)}})] = 2^k$. Thus, $x^{2^k} - \zeta_{2^{e-(i+k)}}$ is irreducible over $F(\zeta_{2^{e-(i+k)}})$. In particular, when $i = 0$ and $k = e-(\nu_{2^\infty_F}+1)$, we get that $\operatorname{min}(\zeta_{2^e}, F) = x^{2^{\scalebox{0.5}{$e-(\nu_{2^\infty_F}-1)$}}} - \zeta_{2^{\scalebox{0.5}{$\nu_{2^\infty_F}-1$}}}$ and $[F(\zeta_{2^e}):F] = 2^{e-(\nu_{2^\infty_F}-1)}$.
			
		\item Suppose that $\zeta_4\notin F$ and $\tq{2^{\scalebox{0.5}{$\nu_{2^\infty_F}$}}} \in F$.  By Definition \ref{def-tfn} and \cite[Lemma 2.9]{marques-quadratic} we obtain that $x^2-\tq{2^{\scalebox{0.5}{$\nu_{2^\infty_F}$}}}x-1$ is the minimal polynomial of $\zeta_{2^{\scalebox{0.5}{$\nu_{2^\infty_F}$}}}$ over $F$. Since $\zeta_{2^{e-i}}$ is a root of irreducible polynomial $x^{2^{e-(\nu_{2^\infty_F}+i)}}-\zeta_{2^{\scalebox{0.5}{$\nu_{2^\infty_F}$}}}$ by Theorem \ref{maincyclic} (5), then $\operatorname{min}(\zeta_{2^{e-i}}, F)=x^{2^{e-(\nu_{2^\infty_F}+i-1)}}-(\tq{2^{\nu_{2^\infty_F}}})x^{2^{e-(\nu_{2^\infty_F}+i)}}-1$. In particular, when $i=0$, we have $\operatorname{min}(\zeta_{2^e}, F)=x^{2^{e-(\nu_{2^\infty_F}-1)}}-\tq{2^{\nu_{2^\infty_F}}}x^{2^{e-\nu_{2^\infty_F}}}-1$ and  $[F(\zeta_{2^{e}}):F]=2^{e-(\nu_{2^\infty_F}-1)}$ as desired.
		\end{itemize}
		\end{proof}
		
			In  the next result, we establish that when $\zeta_4\in F$, the field $F(\zeta_{2^e})$ coincides with the field $F({\tr{2^e}})$.
	We note that in this result we do not have a restriction on the degree of the extension contrary to Theorem \ref{maincyclic}.
	\begin{corollary}\label{equalextension}
		If $\zeta_4\in F$, then $F(\zeta_{2^e})=F(\ts{2^e})=F(\tq{2^e})$. 
	\end{corollary}
	
	\begin{proof}
		Suppose that $\zeta_4\in F$. 
		When $\zeta_{2^e}\in F$, the result is clear, since ${\ts{2^e}}$ and ${\tq{2^e}}$ are both in $F$. 
Now suppose that $\zeta_{2^e}\notin F$. Clearly, we have $F({\ts{2^e}})\subseteq F(\zeta_{2^e})$. To prove the reverse inclusion, we assume by contradiction that $F({\ts{2^e}})\subsetneq F(\zeta_{2^e})$. This implies that $[F(\zeta_{2^e}):F({\ts{2^e}})]=2$, since $\zeta_{2^e}\notin F({\ts{2^e}})$ and it is a root of the polynomial $x^2-{\ts{2^e}}x+1$ with coefficient in $F({\ts{2^e}})$. By \cite[Lemma 2.9]{marques-quadratic}, we find that $o_{F({\ts{2^e}})}(\zeta_{2^e})=2^{e-1}$. However, this implies that $\zeta_4\in F(\zeta_{2^e})\setminus F({\ts{2^e}})$, which contradicts the assumption that $\zeta_4\in F$. Hence, we conclude that $F({\ts{2^e}})=F(\zeta_{2^e})$.  Similarly, we can prove that  $F(\zeta_{2^e})=F({\tq{2^e}})$. Therefore, we have established the desired result.
		
	\end{proof}

			\section{The subextensions \texorpdfstring{$F(\tr{2^e})$}{Lg} of \texorpdfstring{$F(\zeta_{2^e})/F$}{Lg}}
We begin this section by establishing that in the non-cyclic case, a $2$-power cyclotomic extension precisely admits three subextensions of codegree $2$, which we will describe.
		\begin{lemma} \label{codegree2}
		 If $F(\zeta_{2^e})/F$ is a non-cyclic extension, then $F(\zeta_{2^e})/F$ admits exactly three distinct subextensions $M$ such that $[F(\zeta_{2^{e}}):M]=2$. More precisely, we have either:
		\begin{itemize}
			\item $M=F(\zeta_{2^{e-1}})= F(\zeta_{2^e})^{\langle \sigma_{5^{2^{e-3}}}\rangle}$,
			\item $M=F(\ts{2^e})=F(\zeta_{2^e})^{\langle \sigma_{-1}\rangle}$ and $F(\zeta_{2^{e}})=M(\zeta_4)= M(\tq{2^e})$, or
			\item $M=F(\tq{2^e})= F(\zeta_{2^e})^{\langle \sigma_{-5^{2^{e-3}}}\rangle}$ and $F(\zeta_{2^{e}})=M(\zeta_4)= M(\ts{2^e})$.
		\end{itemize}
		In particular, $\langle \sigma_{-1},\sigma_{-5^{2^{e-3}}} \rangle \subseteq \operatorname{Gal}(F(\zeta_{2^e})/F)$, and $|\langle \sigma_{-1},\sigma_{-5^{2^{e-3}}}\rangle|=4$ 
		\end{lemma}
		\begin{proof}
			Suppose that $\operatorname{Gal}(F(\zeta_{2^e})/F)$ is not cyclic. By  \cite[Theorem 2.3]{conradcyclotomic} along with \cite[Lemma 2.1]{conradcyclotomic}, we have that  $\operatorname{Gal}(F(\zeta_{2^e})/F)$ is isomorphic to a subgroup of $\mathbb{Z}/ 2 \mathbb{Z} \times \mathbb{Z}/2^{e-2} \mathbb{Z}$. According to the fundamental theorem of Galois theory, we obtain three distinct subextensions of codegree $2$ corresponding to the distinct invariants subfields under the three subgroups of order $2$ in $\operatorname{Gal}(F(\zeta_{2^e})/F)$.
			Using \cite[Lemma 2.9]{marques-quadratic}, we conclude that $M$ can take one of the following three forms;  $F(\zeta_{2^{e-1}})$, $F(\ts{2^e})$ and $F(\tq{2^e})$. These represent the three distinct possibilities for the subextension of codegree 2.
			Moreover, we have that $F({\ts{2^e}})({\tq{2^e}})= F({\tq{2^e}})({\ts{2^e}})=F(\zeta_{2^e})$ since  $[F(\zeta_{2^e}):F({\ts{2^e}})]=[F(\zeta_{2^e}):F({\tq{2^e}})]=2$, ${\tq{2^e}} \notin  F({\ts{2^e}})$ and ${\ts{2^e}} \notin F({\tq{2^e}})$.
			
	Moreover,  let us assume $M\in \{F({\ts{2^e}}), F({\tq{2^e}})\}$. We have $\zeta_4 \in F(\zeta_{2^e})\setminus M$ since $o_M(\zeta_{2^e})=2^{e-1}$ by \cite[Lemma 2.9]{marques-quadratic}. Consequently, we obtain $M(\zeta_4)\subseteq F(\zeta_{2^e})$. Considering $[M(\zeta_4):M]=2$ and $[F(\zeta_{2^e}):M]=2$, we can conclude that $M(\zeta_4)=F(\zeta_{2^e})$.
			
			Furthermore, by observing that $5^{2^{e-3}}\equiv 1+2^{e-1} \mod 2^e$, we can easily check that
		$F(\zeta_{2^{e-1}})= F(\zeta_{2^e})^{\langle \sigma_{5^{2^{e-3}}}\rangle}$, $F(\zeta_{2^e})^{\langle \sigma_{-1}\rangle}= F({\ts{2^e}})$ and $F(\zeta_{2^e})^{\langle \sigma_{-5^{2^{e-3}}}\rangle}= F({\tq{2^e}})$. The rest of the proof is clear. 
		\end{proof}
			
		The following result explore the structure of the subsextensions $F({\tr{2^e}})/F$ of  $F(\zeta_{2^e})/F$. We establish that these extensions are cyclic with a unique subextension of codegree $2$, namely $F(\ts{2^{e-1}})$. Moreover, we will determine the $2$-tower decomposition of $F({\tr{2^e}})/F$, which is crucial for determining the $2$-tower decomposition of $F(\zeta_{2^e})/F$ in the non-cyclic case. 
\begin{theorem} \label{te}
		Suppose that $e>\nu_{2^\infty_F}$. The following assertions hold.
	\begin{enumerate}
		\item  $F({\ts{2^e}})/F$ and $F({\tq{2^e}})/F$ are cyclic extensions.
		\item The extension $F({\tr{2^e}})/F$ admits a unique subextension $M$ with a codegree of $2$, that is, $M = F(\ts{2^{e-1}})$. Furthermore, the minimal polynomial of ${\tr{2^e}}$ over $F(\ts{2^{e-1}})$ is given by $x^2-(\ts{2^{e-1}}\pm2)$.
			\item $F({\tr{2^e}})/F$ can be uniquely expressed as a $2$-tower decomposition: 
		$$F(\tr{2^e})/F(\ts{2^{e-1}})/\cdots/F(\ts{2^{\scalebox{0.5}{$\nu_{2^\infty_F}^++1$}}})/F.$$ 
		Moreover, we have $[F(\tr{2^e}):F]=2^{e-\nu_{2^\infty_F}^+}$.
			\item We have that 
		\begin{itemize} 
			\item $\operatorname{Gal}(F({\ts{2^e}})/F)$ is isomorphic to a subgroup of $U_{2^e}/\langle [-1] \rangle$. More precisely, $\operatorname{Gal}(F({\ts{2^e}})/F)= \langle [\sigma_{5^{2^{\scalebox{0.5}{$\nu_{2^\infty_F}^+-2$}}}}]_{\langle \sigma_{-1}\rangle} \rangle$.
			\item $\operatorname{Gal}(F({\tq{2^e}})/F)$ is isomorphic to a subgroup of $U_{2^e}/\langle [-5^{2^{e-3}}] \rangle$. More precisely, $\operatorname{Gal}(F({\tq{2^e}})/F)= \langle [\sigma_{5^{2^{\scalebox{0.5}{$\nu_{2^\infty_F}^+-2$}}}}]_{\langle \sigma_{-5^{2^{e-3}}} \rangle} \rangle.$
		\end{itemize} 
	where $[\sigma]_G$ is the equivalence class of $\sigma$ an element of a group $G$ in the quotient group $G\big /H$ where $H$ is a normal subgroup of $G$.
	\end{enumerate}
\end{theorem}
	\begin{proof}
	\begin{enumerate} 
	\item If $[F({\tr{2^e}}):F]\leq 2$, the result follows easily since $\operatorname{Gal}(F({\tr{2^e}})/F)$ is a cyclic group. Now, we suppose $[F({\tr{2^e}}):F]>2$. Therefore, we also have $[F(\zeta_{2^e}):F]>2$ since $F({\tr{2^e}})\subseteq F(\zeta_{2^e})$. When $F(\zeta_{2^e})/F$ is cyclic, the result follows easily from Theorem \ref{maincyclic} since $[F(\zeta_{2^e}):F]>2$.
	
	We suppose that $F(\zeta_{2^e})/F$ is non-cyclic. By \cite[Theorem 2.3]{conradcyclotomic}, we know that $\operatorname{Gal}(F(\zeta_{2^e})/F)$ is isomorphic to a subgroup of $U_{2^e}$. Also, from Lemma \ref{codegree2}, we know that $[F(\zeta_{2^e}): F({\ts{2^e}})]=[F(\zeta_{2^e}):F({\tq{2^e}})]=2$, $F(\zeta_{2^e})^{\langle \sigma_{-1}\rangle}= F({\ts{2^e}})$ and $F(\zeta_{2^e})^{\langle \sigma_{-5^{2^{e-3}}}\rangle}= F({\tq{2^e}})$.  
	Since $\nu_{2^\infty_F}<e$, we know $\zeta_{2^e}\notin F({\ts{2^e}})$. 
 Now set $G:=\operatorname{Gal}(F(\zeta_{2^e})/F)$.	 Since $\langle \sigma_{-1}\rangle$ is a normal subgroup of $G$, then by \cite[Chapter 14, Theorem 14 (4)]{dummit} we have $F({\ts{2^e}})/F$ is a Galois extension with $\operatorname{Gal}(F({\ts{2^e}})/F)\cong G\big /\langle \sigma_{-1}\rangle $ isomorphic to a subgroup of $U_{2^e}/\langle [-1]\rangle \cong \mathbb{Z}/2^{e-2}\mathbb{Z}$. Thus, $F({\ts{2^e}})/F$ is cyclic.
	Similarly, we find that $\operatorname{Gal}(F({\tq{2^e}})/F)\cong G\big /\langle \sigma_{-5^{2^{e-3}}}\rangle $ is isomorphic to a subgroup of $U_{2^e} \big / \langle [-5^{2^{e-3}}]\rangle $ establishing that $F({\tq{2^e}})/F$ is a cyclic extension.  
	Thus, the proof is complete.

		
\item Since $e > \nu_{2^\infty_F}$ by assumption, it follows from Lemma \ref{lem} that $[F(\zeta_{2^e}):F] > 2$. Consequently, $[F(\tr{2^e}):F] \geq 2$. 
We start  by establishing that $F(\ts{2^{e-1}})/F$ is the only subextension of codegree $2$ in $F({\ts{2^e}})/F$ (resp. $F({\tq{2^e}})/F$). When $F(\zeta_{2^e})/F$ is cyclic, the result following directly from Theorem \ref{maincyclic}. 
		
	Now let us assume that $F(\zeta_{2^e})/F$ is non-cyclic. By Lemma \ref{codegree2}, we have $[F(\zeta_{2^e}):F({\ts{2^e}})]=[F(\zeta_{2^e}):F({\tq{2^e}})]=2$ and $\langle \sigma_{-1},\sigma_{-5^{2^{e-3}}} \rangle \subseteq \operatorname{Gal}(F(\zeta_{2^e})/F)$. 
	Also, according to $(1)$ above, we know that both $F({\ts{2^e}})/F$ and
		$F({\tq{2^e}})/F$ are cyclic extensions. Set $G:=\operatorname{Gal}(F(\zeta_{2^e})/F)$.
		Considering the action of the Galois group on ${\ts{2^e}}$ and $\ts{2^{e-1}}$, we have:
\begin{alignat*}{2}
				[\sigma_{5^{2^{e-3}}}]_{\langle \sigma_{-1}\rangle}({\ts{2^e}}) = \sigma_{5^{2^{e-3}}}({\ts{2^e}}) = -(\zeta_{2^{e}}+\zeta_{2^{e}}^{-1})= -{\ts{2^e}}
			\end{alignat*}
				and
				\begin{alignat*}{2}
				[\sigma_{5^{2^{e-3}}}]_{\langle \sigma_{-1}\rangle}(\ts{2^{e-1}}) = \sigma_{5^{2^{e-3}}}(\ts{2^{e-1}}) &=\zeta_{2^{e-1}}+\zeta_{2^{e-1}}^{-1}=\ts{2^{e-1}}.
			\end{alignat*}
		
		That implies that $F(\ts{2^{e-1}})=F({\ts{2^e}})^{\langle[\sigma_{5^{2^{e-3}}}]_{\langle \sigma_{-1}\rangle} \rangle}$. Since $\langle[\sigma_{5^{2^{e-3}}}]_{\langle \sigma_{-1}\rangle} \rangle$ has order $2$ in $G\big /\langle \sigma_{-1}\rangle$, we conclude that $[F({\ts{2^e}}):F(\ts{2^{e-1}})]=2$. By similar reasoning, we can prove that $F(\ts{2^{e-1}})=F({\tq{2^e}})^{\langle[\sigma_{5^{2^{e-3}}}]_{\langle \sigma_{-5^{2^{e-3}}}\rangle} \rangle}$, and once again, $[F({\tq{2^e}}):F(\ts{2^{e-1}})]=2$. Thus, we have shown that $F(\ts{2^{e-1}})$ is the unique subextension of codegree $2$ in both $F({\ts{2^e}})/F$ and $F({\tq{2^e}})/F$. The rest of the statement follows easily from this. 
		

	\item By $(2)$ above, we know that $F({\tr{2^e}})/F$ has a unique subextension $F(\ts{2^{e-1}})/F$ of codegree $2$. To establish the result, it is enough to obtain the $2$-tower decomposition of $F({\ts{2^e}})/F$, for any $e\geq \nu_{2^\infty_F}^++1$. We will proceed by induction on $e$ to prove this. The initial case, $e=\nu_{2^\infty_F}^++1$, is deduced from Definition \ref{def-tfn} and $(2)$ above. The induction step also follows from $(2)$ completing the proof. The multiplicativity property of the degree of a tower of field extensions gives us the formula for $[F({\tr{2^e}}):F]$.	

	\item In the proof of $(1)$ above, we established that $\operatorname{Gal}(F({\ts{2^e}})/F)$(resp. $\operatorname{Gal}(F({\tq{2^e}})/F)$) is isomorphic to a subgroup of $U_{2^e}/\langle [-1] \rangle$ (resp.  $U_{2^e}/\langle [-5^{2^{e-3}}] \rangle$). Since by $(3)$  above, we know that $[F({\ts{2^e}}):F]=2^{e-\nu_{2^\infty_F}^+}$ and the only subgroup of $U_{2^e}/\langle [-1] \rangle$ of order $2^{e-\nu_{2^\infty_F}^+}$ is the group generated by $\big[5^{2^{\scalebox{0.5}{$\nu_{2^\infty_F}^+-2$}}}\big]_{[-1]}$,  $\operatorname{Gal}(F({\ts{2^e}})/F)= \langle [\sigma_{5^{2^{\scalebox{0.5}{$\nu_{2^\infty_F}^+-2$}}}}]_{\langle \sigma_{-1}\rangle} \rangle$. 
A similar argument gives us $\operatorname{Gal}(F({\tq{2^e}})/F)= \langle [\sigma_{5^{2^{\scalebox{0.5}{$\nu_{2^\infty_F}^+-2$}}}}]_{\langle \sigma_{-5^{2^{e-3}}} \rangle} \rangle.$ 

\end{enumerate}	
\end{proof}

	By a straightforward induction, we can deduce the following result from the previous result. 
	\begin{corollary}
		Let $\nu_{2^\infty_F} < e$. In the field extension $F(\ts{2^e})/F$ (resp. $F(\tq{2^e})/F$), there exists a unique intermediate subextension $M$ with a codegree of $2^{e-k}$, for each $k\in \llbracket 1,\nu_{2^\infty_F}^+\rrbracket$. Specifically, we have $M = F(\ts{2^{k}})$ for each $k\in \llbracket 1,\nu_{2^\infty_F}^+\rrbracket$. 
	\end{corollary}
		We can establish that $\operatorname{Gal}(\mathbb{Q}(\zeta_{2^e})/\mathbb{Q})$ is isomorphic to the multiplicative group $U_{2^e}$, as outlined in \cite[Theorem 3.1]{conradcyclotomic}. Thus, we draw the following corollary from Theorem \ref{te}.
	\begin{corollary}
		The Galois group $\operatorname{Gal}(\mathbb{Q}({\tr{2^e}})/\mathbb{Q})$ is isomorphic to the cyclic group in $U_{2^e}$. More precisely, we have $\operatorname{Gal}(\mathbb{Q}({\ts{2^e}})/\mathbb{Q}) = \langle [\sigma_{5}]_{\langle \sigma_{-1} \rangle} \rangle$ and $\operatorname{Gal}(\mathbb{Q}({\tq{2^e}})/\mathbb{Q}) = \langle [\sigma_{5}]_{\langle \sigma_{-5^{2^{e-3}}} \rangle} \rangle$, where $[\sigma]_G$ is the equivalence class of $\sigma$ an element of a group $G$ in the quotient group $G\big /H$ where $H$ is a normal subgroup of $G$.
	\end{corollary}
	
\section{Non-cyclic $2$-power cyclotomic extensions}
In this section, we provide conditions under which a $2$-power cyclotomic extension is non-cyclic.
\begin{theorem}\label{noncyclic-theorem}
		Suppose that $e>\nu_{2^\infty_F}$. 
	The following assertions are equivalent:
	\begin{enumerate}
		\item $\operatorname{Gal}(F(\zeta_{2^e})/F)$ is a non-cyclic group.
			\item $\langle \sigma_{-1}, \sigma_{5^{2^{e-3}}}\rangle \subseteq \operatorname{Gal}(F(\zeta_{2^{e}})/F)$ and $\operatorname{o}(\langle \sigma_{-1}, \sigma_{5^{2^{e-3}}}\rangle ) =4$.
		\item $F(\zeta_{2^e})/F$ admits exactly three distinct subextensions $M$ such that $[F(\zeta_{2^{e}}):M]=2$. More precisely, $M\in \{ F(\zeta_{2^{e-1}}) , F(\ts{2^e}), F(\tq{2^e})\}$, (see also Lemma \ref{codegree2}).
		\item We have $2(e-\nu_{2^\infty_F})+1$, $2$-tower decompositions of $F(\zeta_{2^e})/F$. They are precisely of the following forms:
		\begin{enumerate}
			\item $F(\zeta_{2^e})/F(\zeta_{2^{e-1}})/\cdots/F(\zeta_{2^{\scalebox{0.5}{$\nu_{2^\infty_F}$}}})/F$,
			\item $F(\zeta_{2^e})/\cdots/F(\zeta_{2^{e-r}})/F(\ts{2^{e-r}})/\cdots/F(\ts{2^{\scalebox{0.5}{$\nu_{2^\infty_F}^++1$}}})/F$,
			\item {\small $F(\zeta_{2^e})/\cdots/F(\zeta_{2^{e-r}})/F(\tq{2^{e-r}})/F(\ts{2^{e-(r+1)}})/\cdots/F(\ts{2^{\scalebox{0.5}{$\nu_{2^\infty_F}^++1$}}})/F$},
		\end{enumerate}
		where $r\in \{0,\cdots, e-(\nu_{2^\infty_F}+1)\}$. 
		\item $\zeta_4\notin F$ and  $\tq{2^{\scalebox{0.5}{$\nu_{2^\infty_F}$}}} \notin F$
	\end{enumerate}
	Moreover, if $F(\zeta_{2^e})/F$ is a non-cyclic extension, then $\operatorname{char} (F)=0$. 
\end{theorem}

\begin{proof} 
	$(1)\Longleftrightarrow (2)$ follows by taking the contrapositive of condition $(1)$ in Theorem \ref{maincyclic} and by the fact that $\langle [-1], [5^{2^{e-3}}]\rangle$ is the only non-cyclic subgroup of order $4$ in $U_{2^e}$. 
	
		$(1) \Longrightarrow (3)$ follows by Lemma \ref{codegree2}.  	
		
	$(3)\Longrightarrow(1)$  and $(4)\Longrightarrow (1)$ follow easily by applying the fundamental theorem of Galois theory and \cite[Theorem 1]{cyclicgroup}.
	
	$(1) \Longrightarrow (4) $ We will prove the result by induction  on the degree $[F(\zeta_{2^e}):F]$. We start with the base case $[F(\zeta_{2^e}):F]=4$.
	
	Since we have proven the equivalence of conditions $(1)$ to $ (3)$, then $F(\zeta_{2^{e}})/F$ has three distinct subextensions namely, $F(\zeta_{2^{e-1}})$, $F(\ts{2^{e}})$, and $F(\tq{2^{e}})$ of codegree $2$. Moreover, since $[F(\zeta_{2^e}):F]=4$, then $[F(\zeta_{2^{e-1}}):F]=[F(\ts{2^{e}}):F]=[F(\tq{2^{e}}):F]=2$ by the multiplicativity of the degree of a tower of field extensions. 
Since $[F(\zeta_{2^e}):F]=4$, then $F(\zeta_{2^{e}})=F(\zeta_{2^{\nu_{2^\infty_F}+1}})$ by Lemma \ref{lem}.
	Consequently, we can write $F(\zeta_{2^{\nu_{2^\infty_F}+1}})/F$ as a $2$-tower $F(\zeta_{2^{\nu_{2^\infty_F}+1}})/F(\zeta_{2^{\nu_{2^\infty_F}}})/F$, or $F(\zeta_{2^{\nu_{2^\infty_F}+1}})/F(\ts{2^{\nu_{2^\infty_F}}})/F$, or $F(\zeta_{2^{\nu_{2^\infty_F}+1}})/F(\tq{2^{\nu_{2^\infty_F}+1}})/F$. Since these towers are distinct, we have $3=2(\nu_{2^\infty_F}+1-\nu_{2^\infty_F})+1$ towers. This concludes the proof of the base case. The induction proof can be completed by combining $(3)$ and Theorem \ref{te}.

	$(5)\Longleftrightarrow (1) $ is the negation of $(6)\Longleftrightarrow (1)$ in Theorem \ref{maincyclic}. \\ 

The final statement follows from the fact that any extension of a finite field is cyclic (see also \cite[Lemma 3.34.]{marques-quadratic}), and \cite[Proposition 7.14]{milne}.

\end{proof}

From Theorem \ref{noncyclic-theorem} we deduce the following results. 
\begin{corollary}
Let $e>\nu_{2^\infty_F}$. If $\operatorname{Gal}(F(\zeta_{2^e})/F))$ is a non-cyclic group, we  have $$[F(\zeta_{2^e}):F]= 2^{e-(\nu_{2^\infty_F}-1)} \ \text{and} \ \operatorname{min}(\zeta_{2^e}, F)=x^{2^{e-(\nu_{2^\infty_F}-1)}}-\ts{2^{\nu_{2^\infty_F}}}x^{2^{e-\nu_{2^\infty_F}}}+1.$$
\end{corollary}
\begin{proof}
	By Theorem \ref{noncyclic-theorem}, we know that 
 $F(\zeta_{2^e})/F$ admits a $2$-tower decomposition in the form $F(\zeta_{2^e})/F(\zeta_{2^{e-1}})/\cdots/F(\zeta_{2^{\nu_{2^\infty_F}}})/F$. Thus, by the multiplicativity property of the degree of a tower of extensions we obtain the degree of $F(\zeta_{2^e})/F$.
 
 Moreover, appliying Theorem \ref{noncyclic-theorem}
  $(5)$ and \cite[Theorem 3.24]{marques-quadratic}, we deduce that $\zeta_{2^{\nu_{2^\infty_F}}}$ is a root of the irreducible polynomial $x^2 - \ts{2^{\nu_{2^\infty_F}}} x + 1$ over $F$.  Since $\zeta_{2^e}^{2^{e-\nu_{2^\infty_F}}} = \zeta_{2^{\nu_{2^\infty_F}}}$, then the minimal polynomial of $\zeta_{2^e}$ over $F$ can be expressed as  $x^{2^{e-(\nu_{2^\infty_F}-1)}}-\ts{2^{\nu_{2^\infty_F}}}x^{2^{e-\nu_{2^\infty_F}}}+1.$
\end{proof}
From the above result, we derive the following corollary. 
\begin{corollary}
Let $e>\nu_{2^\infty_F}$.  If 	$\operatorname{Gal}(F(\zeta_{2^e})/F)$ is a non-cyclic group, then $$\operatorname{Gal}(F(\zeta_{2^e})/F) = \langle \sigma_{-1} , \sigma_{5^{ 2^{\nu_{2^\infty_F}-2}}}  \rangle \cong \mathbb{Z}/2\mathbb{Z}\times \mathbb{Z}/2^{e-\nu_{2^\infty_F}}\mathbb{Z}.$$
\end{corollary}


 In the next lemma, we characterize when an extension generated by $\zeta_8$ is non-cyclic. 
 
\begin{lemma}\label{notcyclic}
The following assertions are equivalent:
	\begin{enumerate}
		\item $\zeta_4\notin F$ and $\nu_{2^\infty_F}=2$;
		\item $\operatorname{Gal}(F(\zeta_{8})/F) \cong U_8$.
	\end{enumerate}
Moreover, if any of the assertions $(1)$-$(3)$ is satisfied then $F$ has characteristic $0$. 
\end{lemma}

\begin{proof}
	$ (1)\Longrightarrow (2)$
	Suppose $\zeta_4\notin F$ and $\nu_{2^\infty_F}=2$.
	Since $\nu_{2^\infty_F}=2$, by the definition of $\nu_{2^\infty_F}$, we obtain $\ts{8}$ and $\tq{
	8}$ are not contained in $F$. Consequently, $[F(\ts{8}):F]=2$. 
	Since $\zeta_{8}$ satisfies the polynomial $x^2-\ts{8}x+1$ over $F(\ts{8})$, we need to prove that $x^2-\ts{8}x+1$ is irreducible. The latter is equivalent to proving  that $\tq{8}\notin F(\ts{8})$ by \cite[Remark 2.13]{marques-quadratic}. By way of contradiction, suppose that $\tq{8}\in F(\ts{8})$. That implies that
	$\zeta_8\in F(\ts{8})$  since $\tq{8}+\ts{8}=2\zeta_8$ which implies $F(\ts{8})= F(\zeta_8)$.
	
	Since $\ts{8}\notin F$ and $\zeta_4\notin F$, then by \cite[Lemma 2.9]{marques-quadratic}, the minimal polynomial of $\zeta_8$ over $F$ is $x^2 - \tq{8} x-1$ and thus we obtain $\tq{8}\in F$. However, this means that $F$ has the property $\mathcal{C}_2$ and $\nu_{2^\infty_F}>2$, which contradicts the assumption. Thus, $[F(\zeta_8):F(\ts{8})]=2$, resulting in $[F(\zeta_8):F]=4$.  As a conclusion, by  \cite[Theorem 2.3]{conradcyclotomic} we obtain the result.


	$(2) \Longrightarrow (1)$ We have $[F(\zeta_{8}):F]=4$ by definition of Galois extension. This means $\zeta_4\notin F$, as otherwise it would imply $[F(\zeta_{8}):F]\leq 2$, contradicting the assumption. Additionally, $\ts{8}\notin F$ since $\ts{8}\in F$ would imply $[F(\zeta_{8}):F]=2$. Hence, we conclude that $\nu_{2^\infty_F}=2$.
The statement on the characteristic follows from the comment before the lemma. 
\end{proof}

	\bibliographystyle{Abbrv}
	
\end{document}